\documentclass[12pt]{amsart}
\usepackage{amsmath,amssymb,amsbsy,amsfonts,amsthm,latexsym,
                        amsopn,amstext,amsxtra,euscript,amscd,mathrsfs,color,bm}

\usepackage{float}
\usepackage[english]{babel}
\usepackage{url}
\usepackage[colorlinks,linkcolor=blue,anchorcolor=blue,citecolor=blue,backref=page]{hyperref}

\usepackage{mathtools}
\usepackage{todonotes}
\usepackage[norefs,nocites]{refcheck}

\renewcommand*{\backref}[1]{}
\renewcommand*{\backrefalt}[4]{%
    \ifcase #1 (Not cited.)%
    \or        (p.\,#2)%
    \else      (pp.\,#2)%
    \fi} 

\newtheorem{theorem}{Theorem}
\newtheorem{lemma}[theorem]{Lemma}

\newtheorem{corollary}[theorem]{Corollary}

\theoremstyle{definition}

\newtheorem{remark}[theorem]{Remark}
\newtheorem{problem}[theorem]{Problem}

\numberwithin{equation}{section}
\numberwithin{theorem}{section}
\numberwithin{table}{section}
\numberwithin{figure}{section}

\allowdisplaybreaks

\definecolor{olive}{rgb}{0.3, 0.4, .1}
\definecolor{dgreen}{rgb}{0.,0.5,0.}

\newcommand{\yann}[1]{{\color{red}#1}}

\def\cC{{\mathcal C}}

\def\cH{{\mathcal H}}

\def\C{\mathbb{C}}
\def\G{\mathbb{G}}
\def\Z{\mathbb{Z}}

\def\Q{\mathbb{Q}}

\def\fp{\mathfrak{p}}
\def\fq{\mathfrak{q}}

 \def\mand{\qquad\mbox{and}\qquad}

\def\vec#1{\mathbf{#1}}
\def\ff{\mathbf{f}}
\def\ov#1{{\overline{#1}}}

\def\fp{\mathfrak p}

\newcommand{\rA}{\ensuremath{\mathscr{A}}}

\def\({\left(}
\def\){\right)}

\def\house#1{{%
    \setbox0=\hbox{$#1$}
    \vrule height \dimexpr\ht0+1.4pt width .5pt depth \dp0\relax
    \vrule height \dimexpr\ht0+1.4pt width \dimexpr\wd0+2pt depth \dimexpr-\ht0-1pt\relax
    \llap{$#1$\kern1pt}
    \vrule height \dimexpr\ht0+1.4pt width .5pt depth \dp0\relax}}

\newenvironment{notation}[0]{%
  \begin{list}%
    {}%
    {\setlength{\itemindent}{0pt}
     \setlength{\labelwidth}{1\parindent}
     \setlength{\labelsep}{\parindent}
     \setlength{\leftmargin}{2\parindent}
     \setlength{\itemsep}{0pt}
     }%
   }%
  {\end{list}}

  {\end{list}}

\newcommand{\g}{\gamma}

\newcommand{\Res}{{\operatorname{Res}}}

\renewcommand{\setminus}{\smallsetminus}

\newcommand{\h}{\mathrm{h}}
\newcommand{\di}{\mathrm{div}}
\newcommand\Gm{\G_{\mathrm{m}}}

\begin{document}

\title[Multiplicative dependence of rational values]
{Multiplicative dependence of rational values modulo approximate finitely generated groups}

\author[A. B\' erczes] {Attila B\' erczes}
\address{Institute of Mathematics, University of Debrecen, H-4010 Debrecen, P.O. BOX 12, Hungary}
\email{berczesa@science.unideb.hu}

\author[Y. Bugeaud] {Yann Bugeaud}
\address{Institut de Recherche Math\' ematique Avanc\' ee, U.M.R. 7501, Universit\' e de Strasbourg et C.N.R.S., 7, rue Ren\' e Descartes, 67084 Strasbourg, France} 

\yann{\address{Institut universitaire de France}}
\email{yann.bugeaud@math.unistra.fr}

\author[K. Gy\H ory]{K\'alm\'an Gy\H ory}
\address{Institute of Mathematics, University of Debrecen, H-4010 Debrecen, P.O. Box 12, Hungary}
\email{gyory@science.unideb.hu}

\author[J. Mello] {Jorge Mello}
\address{Department of Mathematics and Statistics, Oakland University, 48307 Michigan, United States.}
\email{jorgedemellojr@oakland.edu}

\author[A. Ostafe] {Alina Ostafe}
\address{School of Mathematics and Statistics, University of New South Wales, Sydney, NSW 2052, Australia}
\email{alina.ostafe@unsw.edu.au}

\author[M. Sha]{Min Sha}
\address{School of Mathematical Sciences, South China Normal University, Guangzhou 510631, China}
\email{min.sha@m.scnu.edu.cn}



\begin{abstract}
In this paper, we establish some finiteness results about the multiplicative dependence of rational values
modulo sets which are `close' (with respect to the Weil height) to division groups of finitely generated multiplicative groups of a number field $K$.
For example, we show that under some conditions on rational functions $f_1, \ldots, f_n\in K(X)$,  
there are only finitely many elements $\alpha \in K$ such that $f_1(\alpha),\ldots,f_n(\alpha)$ are multiplicatively dependent modulo such sets. 
\end{abstract}

\maketitle


\section{Introduction}

\subsection{Motivation}

Given  non-zero complex numbers $\alpha_1, \ldots, \alpha_n \in \C^*$, we say that they are \textit{multiplicatively dependent} if there exist integers $k_1,\ldots,k_n \in \Z$, not all zero, such that
\begin{equation*}
\alpha_1^{k_1}\cdots \alpha_n^{k_n} = 1;
\end{equation*}
and we say that they are \textit{multiplicatively dependent modulo $G$}, where $G$ is a subset of $\C^*$,
if there exist integers $k_1,\ldots,k_n \in \Z$, not all zero, such that
\begin{equation*}
\alpha_1^{k_1}\cdots \alpha_n^{k_n} \in G.
\end{equation*}

Multiplicative dependence of algebraic numbers has been studied for a long time and still very actively;
see, for instance, \cite{BCMOS, BaS, BMZ, DS2016,KSSS, LvdP, OSSZ1, OSSZ2, PSSS, vdPL,  Stew} and the references therein.
The authors in \cite{OSSZ2} have studied the multiplicative
dependenc of elements in an orbit of an algebraic dynamical system,
and recently in \cite{BOSS} this has been extended to the more general setting of multiplicative dependence modulo a finitely generated multiplicative group.

In this paper, we want to study multiplicative dependence among values of rational functions modulo  sets which can be
roughly described as  {\it approximate division groups\/} of finitely generated groups $\Gamma$, denoted by $\Gamma^{\di}_\varepsilon$ (which is defined in the next section), that is, sets which are not far with respect to the Weil height from the division group of a finitely generated multiplicative group of a number field.

The motivation also partly comes from the study of points on subvarieties of tori.
Let  $\Gm$ be  the multiplicative algebraic group over the complex numbers $\C$, that is $\Gm = \C^*$ endowed with the multiplicative group law.
Intersection of varieties in $\Gm^n$ with sets of the type $\Gamma^{\di}_\varepsilon$ falls within two conjectures, the {\it Mordell--Lang conjecture\/} on intersection of varieties with finitely generated
subgroups and the {\it Bogomolov conjecture\/} which is about the discreteness of the set of
points of bounded height in a variety.
This direction has been extensively studied over several decades,
see~\cite{AlSm,AmVia,BEGP,BS,BoZan,Ev,Lau,Lia,Martinez,Maz,Pon,Poo,Rem,Sch1}
and references therein, which in particular give precise
quantitative results about the intersection of varieties with $\Gamma^{\di}_\varepsilon$.

Here, in some sense, instead of assuming each coordinate of a point is from $\Gamma^{\di}_\varepsilon$,
we impose that the coordinates of a point multiplicatively generate an element in $\Gamma^{\di}_\varepsilon$.

\subsection{Notations}
\label{subsection:notation}
Throughout the paper, we use the following notations:
\begin{notation}
 \item[\textbullet]  $K$ is a number field.
 \item[\textbullet]  $\ov{K}$ is an algebraic closure of $K$.
 \item[\textbullet] $S$ is a finite set of places of $K$ containing all the infinite places.
 \item[\textbullet] $O_S$ is the ring of $S$-integers of $K$.
 \item[\textbullet] $R^*$ is the unit group of a ring $R$.
 \item[\textbullet]  $\Gamma$ is  a finitely generated subgroup of $K^*$.
 \item[\textbullet]  For $A \subseteq K^*$,
 $$A^{\di}:=\{\alpha \in \ov{K} : \alpha^m \in A \text{ for some integer } m \ge 1\}.$$
  \item[\textbullet] For $\varepsilon > 0$, $$\Gamma^{\di}_\varepsilon := \{\alpha \beta : \alpha \in \Gamma^{\di}, \beta \in \ov{K}^* \text{ with } \h(\beta) \leq \varepsilon \}.$$
 \end{notation}

Here $\h(\cdot)$ stands for the absolute logarithmic Weil height function.
The set $\Gamma^\di$ is called the division group of $\Gamma$.

In addition, let $M_K$ be the set of places of $K$, $M_K^\infty$ the set of infinite places of $K$,
and $M_K^0 = M_K \setminus M_K^\infty$.

\subsection{Main results}
In this section we state the main results proved in this paper. Informally, our results can be summarised as follows: given $f_1, f_2, \ldots, f_n \in K[X]$ satisfying some natural conditions (some results hold for rational functions as well), we prove finiteness of the sets of
\begin{itemize}
\item $\alpha\in K$ such that $f_1(\alpha), \ldots, f_n(\alpha)$ are multiplicatively dependent modulo $\Gamma^{\di}_\varepsilon$;
\item $\alpha\in\Gamma^{\di}_\varepsilon$ such that $f_1(\alpha), \ldots, f_n(\alpha)$ are multiplicatively dependent modulo $\Gamma^{\di}_\varepsilon$.
\end{itemize}
We would like to unify these results and have a finiteness result for the set of $\alpha\in K(\Gamma^\di)$ satisfying the conclusion above, and thus we conclude this section with an open problem in this direction.  

We now formally state our results.

\begin{theorem}
\label{thm:multdep-poly2}
Let $f_1, f_2, \ldots, f_n \in K[X]$ be pairwise coprime polynomials.
Assume that each of them has at least two distinct roots.
Then, for every $\varepsilon > 0$ there are only finitely many elements $\alpha \in K$ such that $f_1(\alpha), \ldots, f_n(\alpha)$ are multiplicatively dependent modulo $\Gamma^{\di}_\varepsilon$.
\end{theorem}

We first remark that in Theorem~\ref{thm:multdep-poly2}, since $\alpha \in K$ and 
the polynomials $f_1, \ldots, f_n$ are in $K[X]$, by Lemma~\ref{lem:K H} below we know that 
``modulo $\Gamma^{\di}_\varepsilon$" can be reduced to modulo a subset of $K^*$ 
which is somehow ``close" to $\Gamma$.

We also remark that in Theorem~\ref{thm:multdep-poly2} the two conditions of  the polynomials being ``pairwise coprime" and ``each of them has at least two distinct roots"  somehow cannot be removed. For example, choosing $f_1 = X(X+2)$, $f_2 = (X+1)(X+2)$, and $f_3,\ldots,f_n$ arbitrary, 
for any $\alpha = 1/(\beta -1)$ with $\beta \in \Gamma$ and $\beta \ne 1, 1/2$ and $f_3(\alpha) \cdots f_n(\alpha) \ne 0$, 
we have 
$$f_1(\alpha)^{-1} f_2(\alpha) \(f_3(\alpha)\cdots f_n(\alpha)\)^0= \beta \in \Gamma.$$ 
In addition, choosing pairwise coprime polynomials $f_1, f_2, \ldots, f_n \in K[X]$ with  $f_1=X - a$ for some $a \in K$, for any $\alpha \in \Gamma$ satisfying $f_2(\alpha+a) \cdots f_n(\alpha+a) \ne 0$,
we have $$f_1(\alpha + a)\big(f_2(\alpha + a)\cdots f_n(\alpha + a)\big)^0 = \alpha \in \Gamma.$$

Using Theorem~\ref{thm:multdep-poly2}, we establish the following result, which holds for rational functions.
For this, we say that non-zero rational functions $f_1,\ldots, f_n \in K(X)$ are   \textit{multiplicatively independent modulo constants} 
if they are multiplicatively independent modulo $K^*$, that is, there is no non-zero integer vector $(k_1,\ldots,k_n)$ such that 
$$
f_1^{k_1} \cdots f_n^{k_n} \in K^*.
$$
In addition, for any rational function $f \in K(X)$, the numerator and denominator of $f$ are meant to be two polynomials $g, h \in K[X]$, 
respectively, such that $f= g / h$ and $\gcd(g,h)=1$.  

\begin{theorem}  
\label{thm:multdep-rat}
Let $f_1, f_2, \ldots, f_n \in K(X)$ be non-constant rational functions such that they are multiplicatively independent modulo constants. 
Assume that for each $f_i, i =1, 2, \ldots, n$, its numerator either has no linear factor or has at least two distinct linear factors over $K$, and so does its denominator. 
Assume further that $f_1, f_2, \ldots, f_n$ have distinct linear factors over $K$ $($if they have$)$. 
Then, for every $\varepsilon > 0$ there are only finitely many elements $\alpha \in K$ such that $f_1(\alpha), \ldots, f_n(\alpha)$ are multiplicatively dependent modulo $\Gamma^{\di}_\varepsilon$.
\end{theorem}

\begin{remark}
If $f_1, f_2, \ldots, f_n \in K(X)$ in Theorem~\ref{thm:multdep-rat} are all monic (that is, both numerator and denominator are monic), then 
the assumption ``they are multiplicatively independent modulo constants" can be replaced by ``they are multiplicatively independent".  
\end{remark}

The following corollary is about multiplicative dependence in orbits of a rational function,
which somehow can be viewed as an extension of \cite[Theorem 1.7]{BOSS}.
For a rational function $f \in K(X)$ and a positive integer $n \ge 1$, let $f^{(n)}$ be the $n$-th compositional iterate of $f$. 
In addition, for any rational function $f \in K(X)$, if both its numerator and denominator have no linear factor over $K$, we say that \textit{$f$ has no linear factor}. 

\begin{corollary}
\label{cor:multdep-ite}
Let $f \in K(X)$ be a non-constant rational function 
satisfying one of the following two conditions: 
\begin{itemize}
\item $f \in K[X]$, $f$ has at least two distinct roots, and $0$ is not a periodic point of $f$; 
\item $f$ has no linear factor.
\end{itemize}
Then, for any $\varepsilon>0$ and any integer $n \ge 1$, there are only finitely many elements $\alpha \in K$
such that $f^{(m+1)}(\alpha)$, $\ldots, f^{(m+n)}(\alpha)$ are multiplicatively dependent modulo $\Gamma^{\di}_\varepsilon$ for some integer $m \ge 0$.
\end{corollary}

When $n=2$ in Theorem~\ref{thm:multdep-poly2}, we can relax the condition of coprimality on the polynomials $f_1$ and $f_2$.

We say that $f_1,\ldots,f_n\in\C(X)$ \textit{multiplicatively generate} a rational function $g$
if there exist integers $k_1,\ldots,k_n \in \Z$, not all zero, such that
$$
f_1^{k_1} \cdots f_n^{k_n} = g.
$$

We have the following result:

\begin{theorem}
\label{thm:multdep-poly1}
Let $f_1,f_2 \in K[X]$ be polynomials of degree at least $2$, each having at least two distinct roots.
Assume that they cannot multiplicatively generate a power of a linear fractional function. 
Then, for any $\varepsilon>0$  there are only finitely many elements $\alpha \in K$ such that $f_1(\alpha)$ and $f_2(\alpha)$ are multiplicatively dependent modulo $\Gamma^{\di}_\varepsilon$.
\end{theorem}

In Theorem~\ref{thm:multdep-poly1}, the condition related to linear fractional function
cannot be removed. See the example below Theorem~\ref{thm:multdep-poly2}. 
Here, we view non-zero constants as linear fractional functions.  

We remark that in Theorem~\ref{thm:multdep-poly1} we can  replace the condition related to linear fractional function with the total number of distinct roots of $f_1$ and $f_2$ which are not common roots being at least three.

We also remark that the results in Theorems~\ref{thm:multdep-poly2} and \ref{thm:multdep-poly1} are both not effective,
due to Lemma~\ref{lem:KGa}.

As a direct consequence of Maurin's theorem \cite[Th{\'e}or{\`e}me 1.2]{Mau} (see \cite{BHMZ} for an effective version),
if $f_1, f_2, \ldots, f_n \in K(X)$ are such that  $X,f_1,\ldots, f_n$ are multiplicatively independent modulo $\Gamma$, then the set
\begin{equation}
\label{eq:G div}
\left\{\alpha \in \Gamma^\di: f_1(\alpha),\ldots,f_n(\alpha)
\textrm{ multiplicatively dependent mod $\Gamma^\di$}\right\}
\end{equation}
is finite (see~\cite[Lemma 3.2]{OstShp1} for more details). This is an effective generalisation of Liardet's theorem \cite[Th{\'e}or{\`e}me 1]{Lia} on division points on curves; see also \cite[Theorem 2.2]{BEGP} for an effective version of Liardet's result.

We remark that by definition, multiplicative dependence modulo $\Gamma$ is equivalent to multiplicative dependence modulo $\Gamma^\di$.

Using \cite[Th{\'e}or{\`e}me 1.10]{Mau11}, which improves \cite[Th{\'e}or{\`e}me 1.7]{Mau},
we are able to extend this conclusion
by enlarging $\Gamma^\di$ to $\Gamma^\di_\varepsilon$ for certain  $\varepsilon >0$ (but in a non-effective manner).

\begin{theorem}
\label{thm:multdep-rat2}
Let $f_1, f_2, \ldots, f_n \in K(X)$ be such that $X,f_1, \ldots, f_n$ are multiplicatively independent modulo $\Gamma$.
Then, there exists a real number $\varepsilon >0$ for which there are only finitely many elements
 $\alpha \in \Gamma^\di_\varepsilon$ such that $f_1(\alpha), \ldots, f_n(\alpha)$ are multiplicatively dependent modulo $\Gamma^\di_\varepsilon$.
\end{theorem}

We end this section with an open problem. We would like to combine Theorem~\ref{thm:multdep-poly2}  with the finiteness of the set~\eqref{eq:G div}, and ask the following question:

\begin{problem}  \label{prob1}
Let $f_1,\ldots,f_n \in K(X)$ be non-zero rational functions. 
Under what conditions is the following set
\begin{equation}
\label{eq:K G div}
\left\{\alpha \in K(\Gamma^\di) : f_1(\alpha),\ldots,f_n(\alpha)
\textrm{ multiplicatively dependent mod $\Gamma^\di$}\right\}
\end{equation}
finite?
\end{problem}

When $\Gamma=\{1\}$, then $\Gamma^\di$ is the set of all roots of unity,  and $K(\Gamma^\di)$ is the 
cyclotomic closure of $K$. In this case, it has been proven in~\cite[Theorem 4.2]{OSSZ2} that if $f_1,\ldots,f_n$ do not multiplicatively generate a power of a linear fractional function, then the set~\eqref{eq:K G div} is finite (in fact, the result holds more generally for the abelian closure of $K$).

When $n=1$, Problem~\ref{prob1} becomes that for a non-zero rational function $f \in K(X)$, under which condition 
the set  $\{\alpha \in K(\Gamma^\di) : \ f(\alpha) \in \Gamma^\di\}$ is finite. This would extend the cyclotomic version of the Hilbert Irreducibility Theorem proved by Dvornicich and Zannier~\cite[Corollary 1]{DZ} in the case when $\Gamma=\{1\}$.
Recall also that a special case of a general conjecture of R\'emond (see \cite[Conjecture 3.4]{Remond}) asserts that there exists a constant $c_\Gamma$ such that 
for any $\alpha \in K(\Gamma^\di) \setminus \Gamma^\di$, $\h(\alpha) \ge c_\Gamma$ 
(see \cite[Conjecture 1.1 (c)]{Pot}, and see \cite[Theorem 1.3]{Amo} for a non-trivial example). 
Clearly, under this conjecture, the finiteness of the set $\{\alpha \in K(\Gamma^\di) : \ f(\alpha) \in \Gamma^\di\}$ implies 
the finiteness of the set $\{\alpha \in K(\Gamma^\di) : \ \h(f(\alpha)) < \varepsilon\}$ for any $\varepsilon < c_\Gamma$. 

\section{Preliminaries}

\subsection{On some intersections with approximate groups and algebraic subgroups}
We define the set $\rA(K,H)$ as the set of nonzero elements in the algebraic number
field $K$ of height at most $H$, that is,
$$
\rA(K,H) =\left\{\alpha \in  K^*:~ \h(\alpha)\le H \right\}.
$$
We note that by Northcott's Theorem the set $ \rA(K, H)$  is a finite set.

We need the following result from \cite[Theorem 2.1]{OstShp}.

\begin{lemma}
\label{lem:K H}
Let $\{g_1, \ldots, g_r\}$ be a set of generators of $\Gamma$, which minimises $H = \max_{i=1, \ldots, r} \h( g_i)$. Then, for every
$\varepsilon> 0$, we have
$$
K^*\cap \Gamma^{\di}_\varepsilon \subseteq \left \{\beta \eta :~ (\beta, \eta) \in
\Gamma \times \rA(K, \varepsilon  + rH)\right\}.
$$
\end{lemma}


As usual, for any non-constant rational function $f \in K(X)$, the degree of $f$ is defined to be the maximum of the degrees of its numerator and denominator.  

The following result is \cite[Theorem 1.2 (a)]{BOSS}.

\begin{lemma}
\label{lem:KGa}
Let $f \in K(X)$ be a rational function of degree $d \ge 2$.
Assume that $f$ is not of the form $a(X-b)^d$ or $a(X-b)^d / (X-c)^d$ with $a,b,c \in K$, $a(b-c)\ne0$, and $d \in \Z$.
Then, the set $\{\alpha \in K: \, f(\alpha) \in \Gamma \}$ is finite.
\end{lemma}

We remark that the result in Lemma~\ref{lem:KGa} is not effective, due to the use of
the  Faltings theorem~\cite{Falt83}  about finiteness of rational points on a curve. See also~\cite[Corollary 2.2]{OstShp} and references therein. 

%

We conclude this section with a  result of Maurin~\cite[Th{\'e}or{\`e}me 1.10]{Mau11}, which we present in our setting of  parametric curves by noticing \cite[Remarque 1.3]{Mau11}.

For this we introduce the following notation: We define $\cH^{[2]}$ to be the union of all algebraic subgroups  in $\G_m^n$ of codimension at least $2$. For $\varepsilon>0$, we let $\cH^{[2]}_\varepsilon$ be defined similarly as in Section~\ref{subsection:notation}, that is,
$$
\cH^{[2]}_\varepsilon=\{\vec{u}\cdot \vec{v} : \vec{u} \in \cH^{[2]}, \vec{v} \in \G_m^n \text{ with } \h(\vec{v}) \leq \varepsilon \}.
$$

We have the following result, which is a special case of \cite[Th{\'e}or{\`e}me 1.10]{Mau11}.

\begin{lemma}
\label{lem:Maurin}
Let $g_1,\ldots,g_r\in K^*$ and $f_1,\ldots,f_n\in K(X)$ be such that $f_1,\ldots,f_n,g_1,\ldots,g_r$ are multiplicatively independent. Let 
$$
\cC=\{(f_1(\alpha),\ldots,f_n(\alpha),g_1,\ldots,g_r) : \alpha \in \ov{K}\} \subset \G_m^{n+r}.
$$
Then there exists a real number $\varepsilon>0$ such that $\cC\cap \cH^{[2]}_\varepsilon$ is finite.
\end{lemma}

\subsection{On some functional properties of rational functions}
We need the following special case of the result of Young~\cite[Corollary~1.2]{Young}, which 
generalises the previous result of Gao~\cite[Theorem~1.4]{Gao}  to multiplicative independence of 
consecutive iterations of rational functions over fields of characteristic zero. 

\begin{lemma}
\label{lem:Young} 
Let $F$ be an arbitrary field of characteristic zero, and let 
$f \in F(X)$ be a rational function of degree $d \ge 2$ which is not of the form $aX^{\pm d}$. 
Then, for any integer $n\ge 1$, the  iterates 
$f^{(1)}(X), \ldots,f^{(n)}(X)$ are multiplicatively independent modulo constants.
\end{lemma}

We also need the following simple lemma. 

\begin{lemma}
\label{lem:linear}
Let $f \in K(X)$ be a rational function such that it has no linear factor. 
Then, for any non-constant rational function $g \in K(X)$, 
 the rational function $f \circ g$  has no linear factor. 
\end{lemma}

\begin{proof}
First, we note that it suffices to prove that for any monic irreducible factor, say $p(X)$, of either the numerator or the denominator of $f$, 
 the rational function $p \circ g$  has no linear factor. 

By contradiction, suppose that  the rational function $p \circ g$  has a linear factor. 
Then, there is an element, say $\alpha$, in $K$ such that $p \circ g (\alpha) = 0$. 

If $g(\alpha)$ is well-defined, then $g(\alpha) \in K$, which means that the polynomial $p$ has a root (that is, $g(\alpha)$) in $K$. 
However, by assumption $p$ is an irreducible polynomial over $K$ of degree at least 2. 
So, we get a contradiction. 

Now, if $g(\alpha)$ is not well-defined, then $\alpha$ is a pole of $g$. 
Write $p = X^d + a_{1}X^{d-1} + \cdots + a_{d-1} X + a_d$ and $g = u/w$ with $u, w \in K[X]$ and $\gcd(u,w)=1$. 
Since $\alpha$ is a pole of $g$, we have $w(\alpha) = 0$. 
Note that $p \circ g = p(u/w) = \frac{1}{w^d}(u^d + a_1 u^{d-1}w + \cdots + a_{d-1}uw^{d-1} + a_d w^d)$. 
Then, since $p \circ g (\alpha) = 0$ and $w(\alpha)=0$, we obtain $u(\alpha) = 0$. 
So, $\alpha$ is a common root of $u$ and $w$, which contradicts with $\gcd(u, w)=1$. 

Therefore,  the rational function $f \circ g$  has no linear factor. 
\end{proof}

\subsection{Generalised Schinzel-Tijdeman theorem}
Another important tool for our results is the following general version, established in \cite{BBGMOS}, of the Schinzel-Tijdeman
theorem~\cite{SchTij}, which extends~\cite[Theorem~2.3]{BEG} and \cite[Lemma 2.8]{BOSS}. We present it in a simplified form, which is sufficient for our purposes.


%
%
%

Let $K$ be a  number field and $S$ a finite subset of $M_K$ containing all the infinite places.  The following theorem is proved in \cite[Theorem 2.2]{BBGMOS}. 

\begin{lemma}
\label{lem:GenS-T}
Let $f(X)=a_0X^n + \cdots +a_n\in O_S[X]$ be a polynomial of degree $n$ and with at least two distinct 
roots.  There is an effectively
computable constant $C(f,K,S)$, depending only on $f$, $K$ and $S$, so that
the following holds: if $b\in O_S^*$ and if the equation
\begin{equation}\label{eq:SchT}
f(x)=by^m \ \ \ in \ \ x,y\in O_S, \ m\in \Z, m \ge 3,
\end{equation}
has a solution $(x,y)$ with $y\ne 0$ and $y \notin O_S^*$, then  
$$
m \leq C(f,K,S). 
$$
\end{lemma}

We remark that, when $f$ has only simple roots, 
the result in Lemma~\ref{lem:GenS-T} 
 has been established in \cite[Lemma 2.8]{BOSS}, 
In addition, when $S$ only consists of infinite places, the result in Lemma~\ref{lem:GenS-T} 
 has been given in \cite[Theorem 10.3]{ST} (choosing $\tau = 0, z=1, \gamma =1, \varepsilon=b$ there).

%
%

\section{Proofs of the main results}

\subsection{Preliminary discussion}
\label{sec:pre}

Let~$S_\Gamma$ be the following set of places of $K$:
\begin{equation*}
  S_\Gamma := M_K^\infty \cup \bigl\{v\in M_K^0 : \,  \text{$v(\gamma) \ne 0$ for some $\g\in\Gamma$} \bigr\},
\end{equation*}
where, as usual, $v(\gamma)$ means the additive valuation of $v$ at $\gamma$.
Note that the set~$S_\Gamma$ is finite, since~$\Gamma$ is finitely generated.

As usual, we say that a polynomial
$$
f(X) =  a_0X^d+ \cdots + a_{d-1}X+ a_d \in K[X]
$$
has bad reduction at~$v\in M_{K}^0$
if either $v(a_i)<0$ for some~$i \ge 1$ or $v(a_0)\ne0$; otherwise
we say it has good reduction at $v$.

Let
$$
\ff=(f_1,\ldots,f_n) \in K[X]^n
$$
be a vector of non-constant polynomials
$$
f_i(X)=  a_{i,0} X^{d_i}+ \cdots +a_{i,d_i-1}X +a_{i,d_i}  \in K[X], \quad i = 1, \ldots, n,
$$
and we define
\begin{equation*}
\begin{split}
 & S_{\ff,\Gamma} \\
 & \quad =  S_\Gamma \cup
   \{v\in M_K^0 : \, \text{at least one of $f_1,\ldots, f_n$ has bad reduction at $v$} \}.
\end{split}
\end{equation*}
Note that $S_{\ff,\Gamma}$ is a finite set.
Moreover, each $f_i \in O_{S_{\ff,\Gamma}}[X]$, and in fact  for any $v\not\in S_{\ff,\Gamma}$ we have
\begin{equation}
\label{eq:qinSf}
v(a_{i,0}) = 0 \mand v(a_{i,j}) \ge 0, \quad  i=1, \ldots, n, \, j=1,\ldots, d_i.
\end{equation}

If 
$$
\ff=(f_1,\ldots,f_n) \in K(X)^n
$$
is a vector of non-constant rational functions, we will use the same notation $S_{\ff,\Gamma}$ for the set including $S_\Gamma$ and all the places $v\in M_K^0$ such that at least one of the numerators or denominators of $f_1,\ldots,f_n$ has bad reduction at $v$.

By definition, we have
$$
O_{S_\Gamma} \subseteq O_{S_{\ff,\Gamma}} \mand \Gamma \subseteq O_{S_\Gamma}^* \subseteq O_{S_{\ff,\Gamma}}^*.
$$
Note that $O_{S_{\ff,\Gamma}}^*$ is also a finitely generated subgroup of $K^*$.
Hence, it suffices to prove the main results by replacing $\Gamma$ with $O_{S_{\ff,\Gamma}}^*$. 
Then, in the sequel we will prove the main results by replacing $\Gamma^{\di}_\varepsilon$ with $O_{S_{\ff,\Gamma,\varepsilon}}^*$, where $S_{\ff,\Gamma,\varepsilon}$ is some finite set of places containing $S_{\ff,\Gamma}$ and depending also on $\varepsilon$.

\subsection{Proof of Theorem~\ref{thm:multdep-poly2}}

Let $\alpha \in K$ be such that there exist integers $k_1,\ldots, k_n$, not all zero such that
$$
f_1(\alpha)^{k_1} \cdots  f_n(\alpha)^{k_n} \in\Gamma^{\di}_\varepsilon.
$$
By Lemma~\ref{lem:K H} there exists $\beta \in O_{S_{\ff,\Gamma}}^*$ and $\eta \in K^*$ with $\h(\eta)\ll_{\varepsilon,\Gamma} 1$ such that
$$
f_1(\alpha)^{k_1} \cdots  f_n(\alpha)^{k_n}= \beta\eta.
$$
Since $\eta \in K^*$ is of bounded height depending only on $\varepsilon$ and $\Gamma$, by Northcott's theorem there are only finitely many such $\eta$. Thus we can enlarge the set $S_{\ff,\Gamma}$ to include all prime ideals that divide the finitely many elements $\eta$. We also include in this larger set the prime ideals outside $S_{\ff,\Gamma}$ that divide the product  $\displaystyle\prod_{1\le i \neq j\le n}$Res$(f_i,f_j)$ of all the resultants of $f_i$ and $f_j$ for $i\ne j$ (we recall that all $\Res(f_i,f_j)$ are $S_{\ff,\Gamma}$-integers),
which are only finitely many. We denote the new set by $S_{\ff,\Gamma,\varepsilon}$ and we note that $S_{\ff,\Gamma,\varepsilon}$ is still a finite set.

By the construction of the set $S_{\ff,\Gamma, \varepsilon}$,  we have 
$$
K^* \cap \Gamma^{\di}_\varepsilon \subseteq O_{S_{\ff,\Gamma,\varepsilon}}^*. 
$$
Thus, it suffices to prove the desired result by replacing $\Gamma^{\di}_\varepsilon$ with $O_{S_{\ff,\Gamma,\varepsilon}}^*$.

Now, we write
\begin{equation}
\label{eq:multrs2}
f_1(\alpha)^{k_1} \cdots  f_n(\alpha)^{k_n}= \gamma,\qquad  \gamma = \beta\eta \in O_{S_{\ff,\Gamma,\varepsilon}}^*.
\end{equation}

 If $n = 1$, since $f_1$ is a polynomial having at least two distinct roots and $O_{S_{\ff,\Gamma,\varepsilon}}^*$ is a finitely generated subgroup, we see that applying  Lemma~\ref{lem:KGa} to $f_1$ and $O_{S_{\ff,\Gamma,\varepsilon}}^*$ 
 gives the desired finiteness result.
 We thus suppose  that $n \geq 2$, and that the result is valid for $n-1$, in order to apply an induction.

We note that if some $k_i=0$, then the desired finiteness of $\alpha \in K$
satisfying~\eqref{eq:multrs2} follows directly from the induction hypothesis.
Hence, we can assume from now on that $k_1 \cdots k_n \ne 0$.

We now complete the proof case by case.

\textbf{Case I}: $\alpha \in  O_{S_{\ff,\Gamma,\varepsilon}}$.

In this case, since $\alpha \in  O_{S_{\ff,\Gamma,\varepsilon}}$ and $f_i\in O_{S_{\ff,\Gamma,\varepsilon}}[X]$ for any $i=1,\ldots,n$, we have
$$
f_1(\alpha), \ldots, f_n(\alpha) \in O_{S_{\ff,\Gamma,\varepsilon}}.
$$
We note that if $f_i(\alpha)\in O_{S_{\ff,\Gamma,\varepsilon}}^*$ for some $i \in \{1,\ldots,n\}$,
then Lemma~\ref{lem:KGa} implies the finiteness of such $\alpha \in K$ satisfying~\eqref{eq:multrs2}.

Thus, we now assume that  $f_1(\alpha),\ldots, f_n(\alpha)\not\in O_{S_{\ff,\Gamma,\varepsilon}}^*$.
This implies that there exists a prime ideal $\fp \not\in S_{\ff,\Gamma,\varepsilon}$ in $ K$ such that $v_\fp(f_1(\alpha))>0$.
 Moreover, since $f_i(\alpha)\in O_{S_{\ff,\Gamma,\varepsilon}}$, 
we have $v_\fp(f_i(\alpha))\ge 0$ for each $i=2, \ldots, n$.

If $k_1k_i > 0$ for each $i =2,\ldots, n$,
without loss of generality we can assume that $k_1,k_2,\ldots, k_n > 0$.
Then, the equation~\eqref{eq:multrs2} implies
$$
k_1v_\fp(f_1(\alpha))+ \cdots +k_nv_\fp(f_n(\alpha)) = 0.
$$
Since $v_\fp(f_1(\alpha))>0$ and $v_\fp(f_i(\alpha))\ge 0$ for each $i =2, \ldots, n$, we obtain a contradiction.

We now assume $k_1k_i < 0$ for some $i \in \{2,\ldots, n\}$.
In this case, without loss of generality, we assume $k_1>0, \ldots, k_m >0$ and $k_{m+1} <0 , \ldots, k_n < 0$ for some positive integer $m$.
So, the equation \eqref{eq:multrs2} becomes
$$
f_1(\alpha)^{k_1} \cdots f_{m}(\alpha)^{k_m} = \gamma f_{m+1}(\alpha)^{-k_{m+1}} \cdots f_{n}(\alpha)^{-k_n}.
$$
Then, since $v_\fp(f_1(\alpha)) > 0$ and $v_\fp(f_i(\alpha))\ge 0$ for each $i=2,\ldots, n$, there must exist
 some $j \in \{m+1, \ldots, n\}$ such that $v_\fp(f_j(\alpha)) > 0$.
 In other words, we have
$$
f_1(\alpha) \equiv f_j(\alpha) \equiv 0  \pmod \fp.
$$
This allows us to conclude that $v_\fp(\Res(f_1,f_j))>0$ (notice that, since $f_1,f_{j} \in O_{S_{\ff,\Gamma,\varepsilon}}[X]$ and $f_1$ and $f_j$ do not have common roots, we have $\Res(f_1,f_j)\in O_{S_{\ff,\Gamma,\varepsilon}}$ and $\Res(f_1,f_j) \ne 0$).
By our construction of the set $S_{\ff,\Gamma,\varepsilon}$, this implies that $\fp \in S_{\ff,\Gamma,\varepsilon}$,
which is a contradiction with the choice of $\fp$ above.
This completes the proof of Case I.

\textbf{Case II}: $\alpha \not\in  O_{S_{\ff,\Gamma,\varepsilon}}$.

In this case, there exists a prime ideal $\fp$ of  the ring of integers of $K$ such that
$$
\fp \not \in S_{\ff,\Gamma,\varepsilon} \quad\text{and}\quad v_\fp(\alpha) < 0.
$$
Let $d_i = \deg f_i, i =1, \ldots,n$.
Then,  using the ultrametric inequality of non-Archimedean valuations and noticing \eqref{eq:qinSf}, we directly have
\begin{equation}
\label{eq:ord fm}
   v_\fp(f_i(\alpha)) = d_i v_\fp(\alpha)\quad\text{for  $i=1,\ldots,n$}.
\end{equation}

Considering valuations in~\eqref{eq:multrs2} and using~\eqref{eq:ord fm} we obtain (since $v_\fp(\gamma)=0$ due to $\gamma\in O_{S_{\ff,\Gamma,\varepsilon}}^*$)
\begin{equation}
\label{eq:exponent}
k_1d_1 + k_2d_2 + \cdots + k_n d_n = 0.
\end{equation}

We view the above identity as a linear Diophantine equation with unknowns $k_1, \ldots, k_n$ in $\Z$.
Then, we have a basis of the integer solutions $(k_1, k_2, \ldots, k_n)$ of the equation \eqref{eq:exponent}, say,
$$
(t_{i,1}, t_{i,2}, \ldots, t_{i,n}), \quad i = 1, \ldots, n-1.
$$

Therefore, $k_1, k_2, \ldots, k_n$ can be expressed as
$$
k_j = \sum_{i=1}^{n-1} s_i t_{i,j},  \quad j=1,\ldots,n,
$$
for some integers $s_1, \ldots, s_{n-1}$.
Substituting this into the equation \eqref{eq:multrs2}, we obtain
\begin{equation}
\label{eq:fst}
 \left( \prod_{j=1}^{n} f_j(\alpha)^{t_{1,j}} \right)^{s_1}  \cdots \left( \prod_{j=1}^{n} f_j(\alpha)^{t_{n-1, j}} \right)^{s_{n-1}}
=   \gamma.
\end{equation}

Now, we let
$$
F(X) = \prod_{j=1}^{n} f_j(X)^{t_{1,j}}, 
$$
where the exponent vector $(t_{1,1}, \ldots, t_{1,n})$ is non-zero by its choice above. 

For any prime $\fq \not\in S_{\ff,\Gamma,\varepsilon}$, if $v_\fq(\alpha) < 0$, then~\eqref{eq:ord fm} holds and
we have
\begin{equation}
\label{eq:Falpha}
\begin{split}
v_\fq(F(\alpha)) & = \sum_{j=1}^{n} t_{1,j}v_\fq(f_j(\alpha)) \\
& =(t_{1,1}d_1 + t_{1,2}d_2 + \cdots + t_{1,n} d_n) v_\fq(\alpha) = 0,
\end{split}
\end{equation}
since $(t_{1,1}, t_{1,2}, \ldots, t_{1,n})$ is a solution to~\eqref{eq:exponent}.

For any prime $\fq \not\in S_{\ff,\Gamma,\varepsilon}$, if $v_\fq(\alpha) \ge 0$, then by \eqref{eq:qinSf}
we have  $v_\fq(f_i(\alpha)) \ge 0$ for each $i = 1, \ldots, n$.

If there exists some prime $\fq \not\in S_{\ff,\Gamma,\varepsilon}$ such that $v_\fq(\alpha) \ge 0$
and moreover $v_\fq(f_i(\alpha)) > 0, v_\fq(f_j(\alpha)) > 0$ for some $i \ne j$,
then by the same discussion as in the last part of Case I we arrive to a contradiction.

If there exists some prime $\fq \not\in S_{\ff,\Gamma,\varepsilon}$ such that $v_\fq(\alpha) \ge 0$
and moreover $v_\fq(f_i(\alpha)) > 0$ for exactly one $i$ for $i=1, \ldots, n$,
say $v_\fq(f_1(\alpha)) > 0$ and $v_\fq(f_i(\alpha)) = 0$ for each $i = 2, \ldots, n$, then by \eqref{eq:fst} we obtain
$$
s_1 t_{1,1} + \cdots + s_{n-1} t_{n-1, 1} = 0,
$$
which however contradicts with $k_1 \ne 0$ because $k_1 = s_1 t_{1,1} + \cdots + s_{n-1} t_{n-1, 1}$.

Hence, we may assume that for any prime $\fq \not\in S_{\ff,\Gamma,\varepsilon}$, if  $v_\fq(\alpha) \ge 0$,
then $v_\fq(f_i(\alpha)) = 0$ for each $i = 1, \ldots, n$.
In this case, we have $v_\fq(F(\alpha)) = 0$.
Combining this with \eqref{eq:Falpha}, we have $v_\fq(F(\alpha)) = 0$ for any prime $\fq \not\in S_{\ff,\Gamma,\varepsilon}$,
and thus  $F(\alpha) \in O_{S_{\ff,\Gamma,\varepsilon}}^*$.
Now, the desired result follows directly from Lemma~\ref{lem:KGa} (which we can apply, since $f_i$, $i=1,\ldots,n$, has at least two distinct roots and they are pairwise coprime, and therefore, $F$ has at least two distinct roots or two distinct poles). 
This completes the proof.

\subsection{Proof of Theorem~\ref{thm:multdep-rat}}

First, we assume that the rational functions $f_1, f_2, \ldots, f_n$ all have no linear factor. 

Let $g_1, \ldots, g_m$ be all the distinct monic irreducible factors (over $K$) in the numerators and denominators of the rational functions $f_{1},  f_{2}, \ldots, f_{n}$. 
So, by assumption, the irreducible polynomials $g_1, \ldots, g_m$ are all of degree at least two. 
Then, for each $f_i, 1\le i \le n$, we can write 
\begin{equation}  \label{eq:figj}
f_i = a_i \prod_{j=1}^{m} g_j^{e_{ij}},\quad a_i\in K^*,
\end{equation}
for some integers $e_{i1}, \ldots, e_{im}$. 

Let $\alpha \in K$ be such that there exist integers $k_1,\ldots, k_n$, not all zero such that
$$
f_1(\alpha)^{k_1} \cdots  f_n(\alpha)^{k_n} \in\Gamma^{\di}_\varepsilon.
$$

As in \eqref{eq:multrs2}, we can write 
\begin{equation}  \label{eq:f1fn}
f_1(\alpha)^{k_1} \cdots  f_n(\alpha)^{k_n}= \gamma,\qquad  \gamma \in O_{S_{\ff,\Gamma,\varepsilon}}^*,
\end{equation}
where the set $S_{\ff,\Gamma,\varepsilon}$ is defined as in the proof of Theorem~\ref{thm:multdep-poly2}, however without including  the prime ideals outside $S_{\ff,\Gamma}$ that divide the product  $\displaystyle\prod_{1\le i \neq j\le n}$Res$(f_i,f_j)$ of all the resultants of $f_i$ and $f_j$ for $i\ne j$, because $f_i$ and $f_j$ might  not be polynomials.

By the discussion in Section~\ref{sec:pre}, we know that $a_i \in O_{S_{\ff,\Gamma,\varepsilon}}^*$ for each $i=1,\ldots,n$. 
Hence, combining \eqref{eq:f1fn} with \eqref{eq:figj}, we get that for some $\gamma^{\prime} \in O_{S_{\ff,\Gamma,\varepsilon}}^*$, 
\begin{equation}   \label{eq:g1gm}
g_1(\alpha)^{k_1e_{11} + \cdots + k_ne_{n1}} \cdots g_m(\alpha)^{k_1e_{1m} + \cdots + k_ne_{nm}} = \gamma^{\prime}. 
\end{equation}

If for each $1\le j \le m$, $k_1e_{1j} + \cdots + k_ne_{nj} = 0$, then this means that $f_1^{k_1} \cdots f_n^{k_n}$ is a constant, 
which contradicts with the assumption that $f_1, \ldots, f_n$ are multiplicatively independent modulo constants. 

So, we must have that $k_1e_{1j} + \cdots + k_ne_{nj} \ne 0$ for some $1 \le j \le m$. 
Then, in view of \eqref{eq:g1gm} and noticing that $g_1, \ldots, g_m$ are pairwise distinct irreducible polynomials of degree at least 2, 
we obtain directly the desired finiteness result by applying Theorem~\ref{thm:multdep-poly2} to the polynomials $g_1, \ldots, g_m$.  
This completes the proof of the case when $f_1, f_2, \ldots, f_n$ all have no linear factor. 

Now, without loss of generality, we assume that for each $f_i, i =1, 2, \ldots, n$,  
 both its numerator and denominator have linear factors. 

Then, for each $f_i, i=1,2,\ldots, n$, we write 
$$
f_i = a_i f_{i1} f_{i2}, \quad a_i\in K^*,
$$
where $f_{i1} \in K(X)$ is monic and only has linear factors, and $f_{i2} \in K(X)$ is monic and only has irredicible factors of degree at least two; 
and moreover, we write 
$$
f_{i1} = \frac{h_{i1}}{h_{i2}},  \quad h_{i1}, h_{i2} \in K[X], \ \gcd(h_{i1}, h_{i2}) = 1. 
$$
By assumption, for each $i=1,2, \ldots, n$, both $h_{i1}$ and $h_{i2}$ have at least two distinct linear factors
and they only have linear factors. 
Moreover, since we have assumed that $f_1, f_2, \ldots, f_n$ have distinct linear factors, we know that $h_{11}, h_{12}, \ldots, h_{n1}, h_{n2}$ are pairwise coprime.  

Let $g_1, \ldots, g_m$ (assume $m \ge 1$) be all the distinct monic irreducible factors (over $K$) in the numerators and denominators of the rational functions $f_{12},  \ldots, f_{n2}$. 

By assumption,  the irreducible polynomials $g_1, \ldots, g_m$ are all of degree at least two. 
So, the polynomials $h_{11}, h_{12}, \ldots, h_{n1}, h_{n2}$, $g_1, \ldots, g_m$ are pairwise coprime. 

Then, for each $f_i, 1\le i \le n$, we can write 
\begin{equation}  \label{eq:fihigj}
f_i = a_i h_{i1} h_{i2}^{-1}\prod_{j=1}^{m} g_j^{e_{ij}},\quad a_i\in K^*,
\end{equation}
for some integers $e_{i1}, \ldots, e_{im}$. 

As in \eqref{eq:g1gm}, combining \eqref{eq:f1fn} with \eqref{eq:fihigj} we can get that 
for some $\gamma^{\prime} \in O_{S_{\ff,\Gamma,\varepsilon}}^*$, 
\begin{equation}   \label{eq:g1h1gm}
\prod_{i=1}^{n} h_{i1}(\alpha)^{k_i} h_{i2}(\alpha)^{-k_i} \cdot \prod_{j=1}^{m}g_j(\alpha)^{k_1e_{1j} + \cdots + k_ne_{nj}}= \gamma^{\prime}. 
\end{equation} 
Then, in view of \eqref{eq:g1h1gm} and noticing that the integers $k_1, \ldots, k_n$ are not all zero, 
we obtain directly the desired finiteness result by applying Theorem~\ref{thm:multdep-poly2} to the polynomials $h_{11}, h_{12}, \ldots, h_{n1}, h_{n2}$, $g_1, \ldots, g_m$. This completes the proof.

\subsection{Proof of Corollary~\ref{cor:multdep-ite}}
First, we assume that $f \in K[X]$ and $0$ is not a periodic point of $f$.
Since $0$ is not a periodic point of $f$, we have that  for any integer $n \ge 1$, $f^{(n)}(0) \ne 0$,
which means that $f^{(n)}$ has non-zero constant term.
So, all the iterates of $f$ are pairwise coprime. 
In addition, since $f$ has at least two distinct roots, it is easy to see that each iterate of $f$ also has at least two distinct roots. 
Hence, by Theorem~\ref{thm:multdep-poly2} we know that
there are only finitely many elements $\beta \in K$ such that
$f^{(1)}(\beta), \ldots, f^{(n)}(\beta)$ are multiplicatively dependent modulo $\Gamma^{\di}_\varepsilon$.

Now, we assume that $f$ has no linear factor. 
Then, by Lemma~\ref{lem:linear}, the iterates $f^{(1)}, \ldots, f^{(n)}$ all have no linear factor. 
Moreover, it follows directly from Lemma~\ref{lem:Young} that the iterates $f^{(1)}, \ldots, f^{(n)}$ are multiplicatively independent modulo constants. 
So, using Theorem~\ref{thm:multdep-rat} we get that there are only finitely many elements $\beta \in K$ such that
$f^{(1)}(\beta), \ldots, f^{(n)}(\beta)$ are multiplicatively dependent modulo $\Gamma^{\di}_\varepsilon$.

So,  for proving the desired result, it suffices to 
 fix such an element $\beta$ 
and show that there are only finitely many $\alpha \in K$ such that $f^{(m)}(\alpha) = \beta$ for some integer $m \ge 0$.
Indeed, this follows directly from \cite[Lemma 2.3]{BOSS} and the well-known fact that
$f$ has only finitely many preperiodic points lying in $K$.

\subsection{Proof of Theorem~\ref{thm:multdep-poly1}}
The proof follows similar ideas as in the proof of~\cite[Theorem 1.7]{BOSS}.

Let $\alpha \in K$ be such that there exist integers $k_1,k_2$, not both zero,
such that
$$
f_1(\alpha)^{k_1}f_2(\alpha)^{k_2}\in\Gamma^{\di}_\varepsilon.
$$
As in the proof of Theorem~\ref{thm:multdep-poly2}, we enlarge the set $S_{\ff,\Gamma}$ (in this case $\ff=(f_1,f_2)$) to a larger set $S_{\ff,\Gamma,\varepsilon}$ such that
\begin{equation}
\label{eq:multrs}
f_1(\alpha)^{k_1}f_2(\alpha)^{k_2}= \gamma \in O_{S_{\ff,\Gamma,\varepsilon}}^*.
\end{equation}

Also, as in the proof of Theorem~\ref{thm:multdep-poly2} we can assume that $k_1 k_2 \ne 0$.
From~\eqref{eq:multrs} and the power saturation of~$O_{S_{\ff,\Gamma,\varepsilon}}^*$ in~$K^*$, we see that
$$
  \gamma = \beta^{\gcd(k_1, k_2)} \quad \textrm{for some $\beta \in O_{S_{\ff,\Gamma,\varepsilon}}^*$}.
$$
This allows us to take the~$\gcd(k_1, k_2)$-root of~\eqref{eq:multrs}, so
without loss of generality we can assume that
$$
  \gcd(k_1, k_2)=1.
$$

We now complete the proof case by case.

\textbf{Case I}: $\alpha \in  O_{S_{\ff,\Gamma,\varepsilon}}$.

In this case, we have $f_1(\alpha),f_2(\alpha)\in O_{S_{\ff,\Gamma,\varepsilon}}$. We note that if $f_1(\alpha)\in O_{S_{\ff,\Gamma,\varepsilon}}^*$ or $f_2(\alpha)\in O_{S_{\ff,\Gamma,\varepsilon}}^*$,
then Lemma~\ref{lem:KGa} implies the finiteness of such $\alpha \in K$ satisfying~\eqref{eq:multrs}.

Thus, we can assume that  $f_1(\alpha),f_2(\alpha)\not\in O_{S_{\ff,\Gamma,\varepsilon}}^*$.
This implies that there exists a prime ideal $\fp$ of $K$ such that the additive valuation $v_\fp(f_1(\alpha))>0$.
Moreover, since $f_2(\alpha)\in O_{S_{\ff,\Gamma,\varepsilon}}$, we have $v_\fp(f_2(\alpha))\ge 0$.

If $k_1k_2 > 0$, then we can assume that $k_1, k_2 > 0$.
In this case, since Equation~\eqref{eq:multrs} implies
$$
k_1v_\fp(f_1(\alpha))+k_2v_\fp(f_2(\alpha))=0,
$$
 we obtain a contradiction by noticing $v_\fp(f_1(\alpha)) > 0$ and $v_\fp(f_2(\alpha)) \ge 0$.

We now assume $k_1k_2<0$. Moreover, we can assume $k_1>0$ and $k_2<0$ (similar discussion applies for $k_1<0$ and $k_2>0$).
Since $\gcd(k_1,k_2)=1$, there exist integers $s,t$ such that
$$
sk_1 + tk_2=1.
$$
Then, using~\eqref{eq:multrs}, we have
\begin{equation}
\label{eq:multab}
\begin{split}
& f_1(\alpha) = f_1(\alpha)^{sk_1 + tk_2} = \gamma^s (f_1(\alpha)^{-t}f_2(\alpha)^s)^{-k_2},  \\
& f_2(\alpha) = f_2(\alpha)^{sk_1 + tk_2} = \gamma^t (f_1(\alpha)^{-t}f_2(\alpha)^s)^{k_1}.
\end{split}
\end{equation}
We note that, since $f_1(\alpha) \in O_{S_{\ff,\Gamma,\varepsilon}}, \gamma \in O_{S_{\ff,\Gamma,\varepsilon}}^*$ and $-k_2>0$, we have  $f_1(\alpha)^{-t}f_2(\alpha)^s \in O_{S_{\ff,\Gamma,\varepsilon}}$.

 If $f_1(\alpha)^{-t}f_2(\alpha)^s \in O_{S_{\ff,\Gamma,\varepsilon}}^*$, then by~\eqref{eq:multab} we obtain that $f_1(\alpha)\in O_{S_{\ff,\Gamma,\varepsilon}}^*$,
 which contradicts our assumption above. Thus, $f_1(\alpha)^{-t}f_2(\alpha)^s \not\in O_{S_{\ff,\Gamma,\varepsilon}}^*$.
 Then, by Lemma~\ref{lem:GenS-T} (with $y=f_1(\alpha)^{-t}f_2(\alpha)^s$ and noticing $\gamma \in O_{S_{\ff,\Gamma,\varepsilon}}^*$),
 the exponent $-k_2$ is bounded above only in terms of $f_1, f_2, K, \Gamma$ and $\varepsilon$. 
Similarly, we obtain that  the exponent $k_1$ is also bounded above only
 in terms of $f_1, f_2, K, \Gamma$ and $\varepsilon$. 
 Hence, in \eqref{eq:multrs} there are only finitely many choices of the two exponents $k_1, k_2$.
Then, fixing $k_1, k_2$ and  applying Lemma~\ref{lem:KGa} to the rational function $f_1^{k_1}f_2^{k_2}$,
 we obtain the desired finiteness result, where we need to use the assumption on   $f_1$ and $f_2$ that they can not multiplicatively generate a power of a linear fractional function.
This completes the proof of Case I.

\textbf{Case II}: $\alpha \not\in  O_{S_{\ff,\Gamma,\varepsilon}}$.

In this case, as in the proof of Theorem~\ref{thm:multdep-poly2}, we can choose a prime ideal $\fp$ of  the ring of integers of $K$ such that
$$
\fp\not \in S_{\ff,\Gamma, \varepsilon} \quad\text{and}\quad v_\fp(\alpha) < 0.
$$
Let $d_i = \deg f_i, i =1, 2$.
Then,  using the ultrametric inequality of non-Archimedean valuations and noticing \eqref{eq:qinSf}, we directly have
\begin{equation}
\label{eq:ord fm 2}
   v_\fp(f_i(\alpha)) = d_i v_\fp(\alpha)\quad\text{for  $i=1,2$.}
\end{equation}
Considering valuations in~\eqref{eq:multrs} and using~\eqref{eq:ord fm 2} we obtain (since $v_\fp(\gamma)=0$ due to $\gamma\in O_{S_{\ff,\Gamma,\varepsilon}}^*$)
$$
k_1d_1+k_2d_2=0.
$$
Since $\gcd(k_1,k_2)=1$, this implies that $k_1 \mid d_2$ and $k_2 \mid d_1$. Thus we can assume that both $k_1$ and $k_2$ are fixed.
Then, as in Case I, the desired finiteness result follows from Lemma~\ref{lem:KGa}.
This completes the proof.

\subsection{Proof of Theorem~\ref{thm:multdep-rat2}}
The proof follows directly from Maurin's result (that is, Lemma~\ref{lem:Maurin}).
Indeed, let $r$ be the rank of $\Gamma$ modulo torsion and let $g_1,\ldots,g_r \in \Gamma$ be its generators,
and thus, they are multiplicatively independent elements. We define the parametric curve
$$
\cC=\{(\alpha, f_1(\alpha),\ldots,f_n(\alpha), g_1,\ldots,g_r) : \, \alpha \in \ov{K}\}\subset \Gm^{n+r+1}.
$$
We choose $\varepsilon$ to be half of the size of the real $\varepsilon$ from Lemma~\ref{lem:Maurin}.

For an element  $\alpha \in \Gamma^\di_\varepsilon$, assume  that  $f_1(\alpha), \ldots, f_n(\alpha)$ are multiplicatively dependent modulo $\Gamma^\di_\varepsilon$.
Since $\alpha \in \Gamma^\di_\varepsilon$, there exist a non-zero vector $(k_0,\ldots,k_r)\in\Z^{r+1}$, $k_0 \ne 0$, such that
$$
\alpha^{k_0}g_1^{k_{1}}\cdots g_r^{k_{r}}=\beta^{k_0}
$$
for some $\beta \in \overline{K}^*$ with $ \h(\beta) \leq \varepsilon$, implying that
\begin{equation}\label{rel1}
\dfrac{\alpha^{k_0}}{\beta^{k_0}}g_1^{k_{1}}\cdots g_r^{k_{r}}=1.
\end{equation}
Moreover, since $f_1(\alpha), \ldots, f_n(\alpha)$ are multiplicatively dependent modulo $\Gamma^\di_\varepsilon$, there exist some positive integer $t$ and a non-zero vector $(\ell_1,\ldots,\ell_{n+r})\in\Z^{n+r}$ such that
$$
f_1(\alpha)^{t\ell_1}\cdots f_n(\alpha)^{t\ell_n}g_1^{\ell_{n+1}}\cdots g_r^{\ell_{n+r}}=\gamma^t
$$
for some $\gamma \in \overline{K}^*$ with $\h(\gamma) \leq \varepsilon$,
implying that (without loss of generality, we assume $\ell_1 \cdots \ell_n \ne 0$)
 \begin{equation}\label{rel2}
 \dfrac{f_1(\alpha)^{t\ell_1}}{\gamma^{t\ell_1/n\ell_1}}\cdots \dfrac{f_n(\alpha)^{t\ell_n}}{\gamma^{t\ell_n/n\ell_n}}g_1^{\ell_{n+1}}\cdots g_r^{\ell_{n+r}}=1.
 \end{equation}
 Therefore, the point
$$
\left(\dfrac{\alpha}{\beta},\frac{f_1(\alpha)}{\gamma^{1/n\ell_1}},\ldots,\frac{f_n(\alpha)}{\gamma^{1/n\ell_n}},g_1,\ldots,g_r\right)
$$
satisfies the multiplicative dependence relations \eqref{rel1} and \eqref{rel2}, which have linearly independent vectors of exponents.
Moreover,
\begin{align*}
&\left(\alpha,f_1(\alpha),\ldots,f_n(\alpha),g_1,\ldots,g_r\right)\\&=\left(\dfrac{\alpha}{\beta},\frac{f_1(\alpha)}{\gamma^{1/n\ell_1}},\ldots,\frac{f_n(\alpha)}{\gamma^{1/n\ell_n}},g_1,\ldots,g_r\right) \cdot (\beta,\gamma^{1/n\ell_1},\ldots ,\gamma^{1/n\ell_n},1,\ldots,1)
\end{align*} 
is a point on $\mathcal{C}$ with
\begin{align*}
\h(\beta,\gamma^{1/n\ell_1},\ldots,\gamma^{1/n\ell_n},1,\ldots,1)
& : =  \h(\beta)+ \h(\gamma^{1/n\ell_1}) + \cdots + \h(\gamma^{1/n\ell_n}) \\
& \leq \h(\beta) + \h(\gamma) \leq 2\varepsilon.
\end{align*}
We also note that by assumption, the functions $X, f_1,\ldots,f_n,g_1,\ldots,g_r$ are multiplicatively independent.
Hence, the desired result follows
directly from Lemma~\ref{lem:Maurin}.

\section*{Acknowledgement}

The authors are grateful to the referee for a careful reading of the manuscript and valuable comments.

The research of Attila B\'erczes was supported in part by grants K128088 and ANN130909 of the Hungarian National Foundation for Scientific Research and by the project EFOP-3.6.1-16-2016-00022 co-financed by the European Union and the European Social Fund. 
K\'alm\'an Gy\H ory was supported in part by grants
K128088 and ANN130909 of the Hungarian National Foundation for Scientific Research.
Jorge Mello and Alina Ostafe were partially supported by the
Australian Research Council Grants DP180100201 and DP200100355. 
Min Sha was supported by the Guangdong Basic and Applied Basic Research Foundation (No. 2022A1515012032) 
and also by the Australian Research Council Grant DE190100888. 
Alina Ostafe also gratefully acknowledges the generosity and hospitality of the Max Planck Institute for Mathematics where parts of her work on this project were developed.


\begin{thebibliography}{99}

\bibitem{AlSm} I. Aliev and C. J. Smyth, {\it Solving algebraic equations in roots
of unity},   Forum Math. {\bf 24} (2012),  641--665.

\bibitem{Amo}
F. Amoroso, \textit{On a conjecture of G. R\'emond}. Ann. Scuola Norm. Sup. Pise Cl. Sci. \textbf{15} (2016), 599--608.

\bibitem{AmVia}
 F. Amoroso and E. Viada, {\it Small points on subvarieties of a torus},
 Duke Math. J.  {\bf 150} (2009), 407--442.

\bibitem{BCMOS}
 F. Barroero, L. Capuano, L. M{\'e}rai, A. Ostafe and M. Sha, 
 \textit{Multiplicative and linear dependence  in finite fields and on elliptic curves  modulo primes},  
  Int. Math. Res. Not. \textbf{2022} (2022), 16094--16137.  

\bibitem{BaS}
 F. Barroero and M. Sha, \textit{Torsion points with multiplicatively dependent coordinates on elliptic curves},
Bull. London Math. Soc. \textbf{52} (2020), 807--815.
 
\bibitem{BBGMOS}
 A. B\' erczes, Y. Bugeaud, K. Gy\H ory, J. Mello, A. Ostafe and M. Sha, \textit{Explicit bounds for the solutions of superelliptic equations over number fields}, Preprint, 2023, \url{https://arxiv.org/abs/2310.09704}.

\bibitem{BEG} A. B{\'e}rczes, J.-H. Evertse and K.  Gy{\" o}ry, \textit{Effective results for hyper- and superelliptic equations over number fields},
Publ. Math. Debrecen {\bf 82} (2013),  727--756.

\bibitem{BEGP} A. B{\' e}rczes, J.-H. Evertse, K. Gy\" ory and C. Pontreau, \textit{Effective results for points on certain subvarieties of tori},
Math. Proc. Cambridge Phil. Soc. \textbf{147} (2009), 69--94.

 \bibitem{BOSS} A. B\' erczes, A. Ostafe, I. E. Shparlinski and J. H. Silverman,
 \textit{Multiplicative dependence among iterated values of rational functions modulo finitely generated groups},
 Int. Math. Res. Not. \textbf{2021} (2021), 9045--9082.

\bibitem{BS} F. Beukers and C. J. Smyth, {\it Cyclotomic points on curves},
Number Theory for the Millenium (Urbana, Illinois, 2000), I, A K Peters, 2002, 67--85.


\bibitem{BHMZ}
E. Bombieri, P. Habegger, D. Masser and U. Zannier, \textit{A note on Maurin's theorem}, Rend. Lincei Mat. Appl. {\bf 21} (2010), 251--260.

\bibitem{BMZ}
E. Bombieri, D. Masser and U. Zannier, \textit{Intersecting a curve with algebraic subgroups
of multiplicative groups}, Int. Math. Res. Not. \textbf{1999} (1999), 1119--1140.


 \bibitem{BoZan} E. Bombieri and U. Zannier, \textit{Algebraic points on subvarieties of $\G_m^n$}, Int. Math. Res. Not.
\textbf{1995} (1995), 333--347.


 


%
%
%

\bibitem{DS2016} A. Dubickas and M. Sha, \textit{Multiplicative dependence of the translations of algebraic numbers}, Rev. Mat. Iberoam. \textbf{34} (2018), 1789--1808.

\bibitem{DZ} R. Dvornicich and U. Zannier, `Cyclotomic diophantine problems (Hilbert
irreducibility and invariant sets for polynomial maps)', {\it Duke Math. J.}, {\bf 139} (2007),
527--554.




 \bibitem{Ev}
 J.-H. Evertse, \textit{Points on subvarieties of tori},
 in: A panorama of number theory or the view from Baker's garden (Z{\"u}rich, 1999), Cambridge University Press,  2002, 214--230.






\bibitem{Falt83}
G. Faltings, \textit{Endlichkeitss\" atze fur abelsche Variet\" aten \" uber Zahlkorpern}, Invent. Math. \textbf{73}  (1983), 349--366.

%

\bibitem{Gao}
S. Gao, \textit{Elements of provable high order in finite fields}, Proc. Am. Math. Soc. \textbf{127} (1999), 1615--1623.



%
%
%




%
%



\bibitem{KSSS} S. V. Konyagin, M. Sha, I. E. Shparlinski and C. L. Stewart, \textit{On the distribution of multiplicatively dependent vectors}, Math. Res. Letters, 30 (2023), 509--540.



\bibitem{Lau} M. Laurent, {\it Equations diophantiennes exponentielles},  Invent. Math. {\bf 78} (1984), 299--327.


\bibitem{Lia} P. Liardet, {\it Sur une conjecture de Serge Lang}, Ast\' erisque {\bf 24}-{\bf 25} (1975), 187--210.

\bibitem{LvdP} J. H. Loxton and A. J. van der Poorten, \textit{Multiplicative dependence in number fields}, Acta Arith. \textbf{42} (1983), 291--302.

\bibitem{Martinez}
C. Martinez, \textit{The number of maximal torsion cosets on subvarieties of tori}, J. Reine Angew. Math. \textbf{755} (2019), 10--126.


\bibitem{Mau} G. Maurin, {\it Courbes alg\' ebriques et \' equations multiplicatives}, Math. Ann. {\bf 341} (2008), 789--824.

\bibitem{Mau11} G. Maurin, \textit{{\'E}quations multiplicatives sur les sous-vari{\'e}t{\'e}s des tores}, Int. Math. Res. Not. \textbf{2011} (2011), 5259--5366.


\bibitem{Maz}
B. Mazur,  {\it Abelian varieties and the Mordell--Lang conjecture},
 Model theory, algebra, and geometry,
Math. Sci. Res. Inst. Publ. \textbf{39}, Cambridge University Press,  2000, 199--227.




\bibitem{OSSZ1} A. Ostafe, M. Sha, I. E. Shparlinski and U. Zannier, \textit{On abelian multiplicatively dependent points on a curve in a torus},
Q. J. Math. \textbf{69} (2018), 391--401.


\bibitem{OSSZ2} A. Ostafe, M. Sha, I. E. Shparlinski and U. Zannier, \textit{On multiplicative dependence of values of rational functions and a generalisation
of the Northcott theorem},  Michigan Math. J. \textbf{68} (2019), 385--407.

\bibitem{OstShp}
A. Ostafe and I. E. Shparlinski, \textit{Orbits of algebraic dynamical systems in subgroups and subfields},
In \textit{Number Theory - Diophantine problems, uniform distribution and applications}, C. Elsholtz and P. Grabner (Eds.), Springer, 2017, 347--368.

\bibitem{OstShp1}
A. Ostafe and I. E. Shparlinski, \textit{On the Skolem problem and some related questions for parametric families of linear recurrence sequences},  Canad. J. Math. (2021), 1--20.


\bibitem{PSSS} F.~Pappalardi, M. Sha, I. E.~Shparlinski and C. L.~Stewart, \textit{On multiplicatively dependent vectors of algebraic numbers},
Trans. Amer. Math. Soc. \textbf{370} (2018), 6221--6244.

\bibitem{Pon} C. Pontreau, {\it A Mordell--Lang plus Bogolomov type result for
curves in $\G_m^2$}, Monatsh. Math. {\bf 157} (2009), 267--281.

\bibitem{Poo} B. Poonen, {\it Mordell--Lang plus Bogomolov}, Invent. Math. {\bf 137} (1999), 413--425.

\bibitem{Pot}
L. Pottmeyer, \textit{Fields generated by finite rank subgroups of $\overline{\Q}^*$}, Int. J. Number Theory \textbf{17} (2021), 1079--1089.

\bibitem{vdPL} A. J.~van~der~Poorten and J. H.~Loxton, \textit{Multiplicative relations in number fields}, Bull. Austral. Math. Soc. \textbf{16} (1977), 83--98.

\bibitem{Rem} G. R\' emond, {\it Sur les sous-vari\' et\' es des tores}, Compos. Math. {\bf 134} (2002), 337--366.

\bibitem{Remond}
G. R\'emond, \textit{G\'en\'eralisations du probl\`eme de Lehmer et applications \`a la conjecture de Zilber-Pink}, Panor. Synth. \textbf{52} (2017), 243--284. 


\bibitem{SchTij} A. Schinzel and R. Tijdeman, \textit{On the equation $y^m = P(x)$}, Acta Arith. \textbf{31} (1976), 199--204.

\bibitem{Sch1} W. M. Schmidt, {\it Heights of points on subvarieties of $\G_m^n$}, In {\it Number Theory} (Paris, 1993--1994),
London Math. Soc. Lecture Note Ser. {\bf 235}, Cambridge University Press, 1996, 157--187.
  
\bibitem{ST}
T. N. Shorey and R. Tijdeman,  \textit{Exponential diophantine equations}, Cambridge University Press, Cambridge, 1986.



%
%
%
%

\bibitem{Stew}
C. L.  Stewart, \textit{On heights of multiplicatively dependent algebraic numbers}, Acta Arith. \textbf{133}  (2008), 97--108.

\bibitem{Young}
M. Young, \textit{On multiplicative independence of rational function iterates}, Monatsh. Math. \textbf{192} (2020), 225--247.



 

\end{thebibliography}
\end{document}